\documentclass[11pt,a4paper]{amsart}

\usepackage{amssymb}
\usepackage{amsthm}
\usepackage{amsmath}
\usepackage{amstext}

\numberwithin{equation}{section}

\theoremstyle{plain}
\newtheorem{thm}{Theorem}
\newtheorem{prop}[thm]{Proposition}

\newtheorem{defn}[thm]{Definition}
\newtheorem*{notation}{Notation}
\newtheorem*{observation}{Observation}
\newtheorem{example}[thm]{Example}

\newcommand{\C}{\mathbb{C}}
\newcommand{\R}{\mathbb{R}}

\newcommand{\N}{\mathbb{N}}

\newcommand{\p}{\partial}
\DeclareMathOperator{\grad}{\p}

\def\({(\!(}
\def\){)\!)}

\title[The \L ojasiewicz exponent]{ The \L ojasiewicz exponent for  weighted
homogeneous polynomials of two real variables }

\author{Ould M Abderrahmane}

\address{D\'eparement de Math\'ematiques, Universit\'e des Sciences,
de Technologie et de Médecine BP. 880, Nouakchott, Mauritanie}

\email{ymoine@ustm.mr} 

\newcommand{\abstracttext}{ The purpose of this paper is to give the exact value of the
{\L}ojasiewicz exponent for an isolated weighted homogeneous
polynomials of two real variaibles in terms of its weights.}

\begin{document}
\begin{abstract} \abstracttext \end{abstract} \maketitle

Let $f \colon (\R^n, 0) \rightarrow (\R, 0)$ be a analytic function.
The {\L}ojasiewicz exponent $L(f)$ of $f$ is by definition
$$
L(f):=\inf\{\lambda> 0\; \,:\;\, \mid\text{grad} \,f\mid\geq
\text{const.} \mid x\mid^{\lambda}\,\text{ near zero }\},
$$
It is well known (see \cite{Bochnak-Risler, LT}) that the
{\L}ojasiewicz exponent can be calculated by means of analytic paths
\begin{equation}\label{0.1}
L(f)=\sup\left\{\frac{\text{ord}(
\text{grad}f(\varphi(t)))}{\text{ord}(\varphi(t))}\; :\; 0\neq
\varphi(t)\in \R\{t\}^n,\; \varphi(0)=0\right\},
\end{equation}
where $\text{ord}(\phi):=\inf_i\{\text{ord}(\phi_i)\}$ for $\phi\in
{\R\{t\}}^n$. By definition, we put $\text{ord}(0) = +\infty$. It is
also known that $L(f) <+\infty$ if and only if $f$ has an isolated
singularity at the origin. The computation or estimation of the \L
ojasiewicz Exponent is a quite interesting problem not only in
geometric analysis but also in singularity theory. For example,
Kuiper-kuo theorem (\cite{kuiper, kuo}) proved that for any integer
$r$ greater than $L( f)$, $f$ is a $C^0$-sufficient, $r$-jet, i.e.,
adding to the function $f$ monomials of order greater than $L(f)$
does note change its topological type. J. Bochnak  and S. \L
ojasiewicz \cite{BL} showed that $C^0$-sufficiency degree of $f$
(i.e., the minimal integer $r$ such that f is $C^0$-sufficient,
$r$-jet) is equal to $[L(f)]+1$, where $[L(f)]$ denotes integral
part of $L(f)$.
\begin{observation}
 Let $f : (\C^n, 0) \to(\C, 0)$ be weighted homogeneous
polynomials with isolated singularities of degree $d$, and let
$w=(w_1,\dots,w_n)$  be the weights of $f$ i.e., $f$ is weighted
homogeneous of type $(d;w)$, recently the author in \cite{OMA} and
S. Brzostowski \cite{SB} prove that the \L ojasiewicz Exponent of
$f$ is precisely equal to
$$
 L(f)=\max_{ i=1}^n (\frac{d}{w_i} - 1).
$$
\end{observation}
Estimates of the \L ojasiewicz Exponent for weighted homogeneous
isolated singularities in the real cases  are in a recent paper by
Haraux and Pham  \cite{AH}.

Motivated by the above observation, we are looking to establish the
\L ojasiewicz Exponent for  the classes of weighted homogeneous
polynomials of two real variables in terms of the weights. To prove
the main result (Theorem \ref{main} below), we recall
 the notion of weighted homogenous filtration introduced by
Paunescu in \cite{LP}. By using it and the generalized Euler
identity, we can compute the {\L}ojasiewicz exponent of weighted
homogeneous polynomials of two real variables.

\begin{notation}
To simplify the notation, we will adopt the following conventions\,:
for a function $g(x, y)$ we denote by $\grad g$ the gradient of $g$
and by $\grad_x g$ the gradient of $g$ with respect to variables
$x$.

Let $\varphi,\,\, \psi \colon(\R^n, 0) \to \R$ be two function
germs. We say that $ \varphi(x)\lesssim \psi(x)$ if there exists a
positive constant $C> 0$ and an open neighborhood $U$ of the  origin
in $\R^n$ such that $\varphi(x)\leq C \; \psi(x)$, for all $x \in
U$. We write $ \varphi(x) \sim \psi(x)$  if $\varphi(x) \lesssim
\psi(x)$ and $\psi(x)\lesssim \varphi(x)$. Finally,
$|\varphi(x)|\ll|\psi(x)|$ (when $x$ tends to $x_0$) means
$\lim_{x\to x_0}\frac{\varphi(x)}{\psi(x)}=0$.

\end{notation}

\bigskip
\section{Weighted  homogeneous filtration, main results
 }
\bigskip

Let $\N$ be the set of nonnegative integers and  $\mathcal{O}_n$
denote the ring of analytic function germs $f \colon (\R^n, 0) \to
(\R, 0)$.

From now, we shall fix a system of positive integers $w=(w_{1},
\dots , w_{n})\in (\N-\{0\})^n$, the weights of variables $x_{i},$
$w(x_{i})=w_{i},$  $1\leq i\leq n$, and a positive integer $d$,
 then a polynomial $f \in\R[x_1,\dots, x_n]$ is called weighted
homogeneous of degree $d$ with respect the weight
$w=(w_1,\dots,w_n)$ (or type $(d; w)$) if $f$ may be written as a
sum of monomials $x_1^{\alpha_1}\cdots x_n^{\alpha_n}$ with
\begin{equation}\label{weight}
\alpha_1 w_1 +\dots +\alpha_n w_n =d.
\end{equation}
 We say that an  analytic function $f\colon (\R^n, 0) \to
(\R, 0)$ is non-degenerate if $\{\frac{\p f}{\p_{x_1}} (x) = \cdots=
\frac{\p f}{\p_{x_n}}(x) = 0\} = \{0\}$ as germs at the origin of
$\R^n$.

We may introduce (see \cite{LP}) the function $\rho(x)=
\left(|x_1|^{\frac{2}{w_1}}+\cdots
+|x_n|^{\frac{2}{w_n}}\right)^{\frac{1}{2}} $. We also consider the
spheres associated to this $\rho$
$$
S_{r}=\{x\in \R^{n}\;:\; \,\rho(x)=r\}, \quad \; r>0.
$$

Here $\cdot$ means the weighted action, with respect to the $\R^*$
action defined below
$$
t\cdot x=(t^{w_1}x_1,\dots, t^{w_n}x_n)
$$

\begin{defn} Using $\rho$,
we define a singular Riemannian metric on $\R^n$ by the following
bilinear form
$$
\langle \rho^{w_i}\frac{\partial}{\partial x_i}\,,\,
\rho^{w_j}\frac{\partial}{\partial x_j} \rangle =
\delta_{i,j}:=\begin{cases}
1 & \text{ if } i=j\\
0 & \text{ if } i\neq j
\end{cases}
$$

 We will denote by $\text{grad}_w$ and  $\parallel\; \parallel_w$, the corresponding gradient and norm
associated with this Riemannian metric  (for more details about
these see \cite{LP}).
\end{defn}

Let $f\in \mathcal{O}_n$. We denote the Taylor expansion of $f$ at
the origin by $\sum c_{\nu}x^{\nu}$. Setting $H_j(x)=\sum
c_{\nu}x^{\nu}$ where the sum is taken over $n$ with $<w, \nu>=
w_1\nu_1+\cdots+w_n\nu_n=j$, we can write the weighted homogeneous
Taylor expansion $f$
$$
f(x)=H_d(x)+H_{d+1}(x)+\cdots\; ; H_d \neq  0.
$$
We call $d$ the  weighted degree of $f$ and  $H_d$ the weighted
initial form of $f$ about the weight. Furthermore, for any $f\in
\mathcal{O}_n$ we get
\begin{equation}\label{1.1}
\|\text{grad}_wf(x)\|_w\lesssim \rho^{d_w(f)}(x),
\end{equation}
where $d_w(f)$ denotes the degree of $f$ with respect to $w$.
Indeed, as all nonzero $x$, we find $\frac{1}{\rho(x)}\cdot x\in
S_1$, moreover, we have $\frac{\partial H_j}{\partial x_i}$  is zero
or a  weighted homogeneous polynomial of degree $d-w_j$, then we
obtain
$$
\|\text{grad}_w H_j(\frac{1}{\rho(x)}\cdot x)\|_w=
\frac{\|\text{grad}_w H_j(x)\|_w}{\rho(x)^j} \lesssim 1.
$$
Therefor,
$$
\parallel\text{grad}_wf(x)\parallel_w\lesssim\sum_{j\geq d_w(f)}\parallel \text{grad}_wH_j(x) \parallel_w\lesssim
\rho^{d_w(f)}(x).
$$

\begin{prop}\label{Hassane}
Let   $f\in \mathcal{O}_n$ be a weighted homogeneous isolated
singularity of type $(d; w)$ at $0\in \R^n $. Then
\begin{equation}\label{1.2}
\|\text{grad}_wf(x)\|_w \gtrsim \rho(x)^d.
\end{equation}
\end{prop}

\begin{proof}
Since $f$ has only  isolated singularity at the origin, then  for
small values of $r$  we have
\begin{equation}\label{1.3}
\| \text{grad}_w f(x)\|_w = \left( \sum_{i=1}^n |\rho^{w_i}(x)
\frac{\partial f}{\partial z_i}(x) |^2 \right)^{\frac{1}{2}}\gtrsim
1, \; \forall x\in S_r.
\end{equation}
 On the other hand, $\frac{\partial f}{\partial x_i}$ is  weighted
 homogeneous of degree $d-w_i$ for $i=1,\dots, n$ and also, $\frac{r}{\rho(x)}\cdot x\in S_r$ for all nonzero $x$.
  Thus, by (\ref{1.3}) we obtain
$$
\| \text{grad}_w f(\frac{r}{\rho(x)}\cdot x)\|_w=r^d\frac{\|
\text{grad}_w f(x)\|_w}{\rho(x)^d}\gtrsim 1.
$$
This completes the proof of the proposition.
\end{proof}


The main result of this paper is the following:

\begin{thm}\label{main}
Let $f\colon (\R^2, 0) \to (\R, 0)$ be non-degenerate weighted
homogeneous polynomial of type $(d; w_1, w_2)$ such that $w_1\ge w
_2$. Then

$$
L(f)=\begin{cases}
 \frac{d-w_1}{w_2} & \text{ if } \;\{\frac{\p f}{\p x_1}(x_1,x_2)=0\}\subset\{x_2=0\},\\
\frac{d}{w_2} -1 & \text{ if not.}
\end{cases}
$$
\end{thm}

\section{Proof of Theorem \ref{main}}

We first note that in the  case where $w_1=w_2$ (i.e, homogenous
filtration), so we can find from the (\ref{1.1}) and (\ref{1.2})
that
$$
\|\text{grad}_wf(x)\|_w \sim \rho(x)^d.
$$
Butt $\|\text{grad}_wf(x)\|_w\sim \parallel x
\parallel\;\parallel\p f\parallel $ and $ \rho(x)^d\sim \parallel x \parallel^{\frac{d}{w_1}}=\parallel x
\parallel^{\frac{d}{w_2}}$.
 Hence, we get
 $$
 L(f)=\frac{d}{w_2} -1=\frac{d-w_1}{w_2}.
 $$
From now, we suppose  that $w_1>w_2$. There are two cases to be
considered.\\

{\bf Case 1.} In this case, we suppose that $\{\frac{\p f}{\p
x_1}(x_1,x_2)=0\}\subset\{x_2=0\}$, take an analytic path $\varphi
(t)=(t^{w_1},t^{w_2})=t\cdot(1,1)$, then from (\ref{0.1}) we get
$$
L(f)\geq \frac{\text{ord}(\grad
f(\varphi(t))}{\text{ord}(\varphi(t))} =\frac{d-w_1}{w_2}.
$$

 Since $f$ defining an isolated singularity at the origin $0\in
\R^n$, there exist the terms

$\varphi \colon (\R^2, 0) \to (\R, 0)$ with  the origin is an
isolated zero of $\p_x\varphi$ i.e., $\{\p_{x_1}\varphi=0\}=\{0\}$
and
$$
f(x_1,x_2)=\varphi(x_1,x_2) \;\text{ or
}\;f(x_1,x_2)=x_2\varphi(x_1,x_2).
$$
For the case where the origin is an isolated zero of $\p_{x_1}f$,
since $\p_{x_1}f$ $w$-form of degree $d-w_1$, it follow from
(\ref{1.3}) that
$$
\parallel \p f\parallel\gtrsim \rho^{d-w_1}\gtrsim (\mid x_1\mid^{1/w_1}+\mid x_2\mid^{1/w_2})^{d-w_1}\gtrsim \parallel
(x_1,x_2)\parallel^{\frac{d-w_1}{w_2}}.
$$
Hence $L(f)\leq \frac{d-w_1}{w_2}$.

This ends the proof of Theorem \ref{main} in the case where the
origin is an isolated zero of $\p_{x_1}f$.

From now we suppose that $f(x_1,x_2)=x_2\varphi(x_1,x_2)$, it is
easy to see that  $\varphi $ is weighted homogeneous of degree
$d-w_2$ (type $(d-w_2; w_1,w_2))$, we have that
 $$
\parallel\p f\parallel^2=\mid x_2\p_{x_1}\varphi\mid^2+\mid\varphi +
x_2\p_{x_2}\varphi\mid^2
$$

Moreover, it follows from the generalized Euler identity that
$$
\varphi(x_1,x_2)=\frac{w_1}{d-w_2}x_1\p_{x_1}\varphi(x_1,x_2)+\frac{w_2}{d-w_2}x_2\p_{x_2}\varphi(x_1,x_2).
$$
Then, we obtain that
$$
\parallel\p f\parallel^2=\parallel(\frac{w_1}{d-w_2}x_1,x_2)
\parallel^2\mid\p_{x_1}\varphi\mid^2
+2\frac{w_1d}{d-w_2}x_1x_2\p_{x_1}\varphi\p_{x_2}\varphi + \mid
\frac{d}{d-w_2}x_2\p_{x_2}\varphi \mid^2.
$$
But, it follows from $w_1>w_2$  that
$$
d_w((\p_{x_1}\varphi)^2)=2(d-w_2-w_1)<d-w_2-w_1+d-w_2-w_2=d_w(\p_{x_1}\varphi
\p_{x_2}\varphi)<d_w((\p_{x_2}\varphi)^2).
$$
Therefore , by the origin is an isolated zero of $\p_{x_1}\varphi$
we get that
$$
\parallel\p f\parallel\gtrsim
\parallel(x_1,x_2)\parallel\rho^{d-w_2-w_1}=\parallel(x_1,x_2)\parallel(\mid x_1\mid^{1/w_1}+\mid
x_2\mid^{1/w_2})^{d-w_2-w_1}\gtrsim\parallel(x_1,x_2)\parallel^{\frac{d-w_1}{w_2}}.
$$
Hence $L(f)\leq \frac{d-w_1}{w_2}$.

This ends the proof of Theorem \ref{main} in the first case.\\

{\bf Case 2.} In this case, we suppose that $\{\frac{\p f}{\p
x_1}(x_1,x_2)=0\}\nsubseteq\{x_2=0\}$. Let $a=(a_1,a_2)\in\{\frac{\p
f}{\p x_1}(x_1,x_2)=0\}$ and $a_2\neq0$, take an analytic path
$\varphi (t)=(t^{w_1}a_1,t^{w_2}2_2)=t\cdot a$, then from
(\ref{0.1}) we get
$$
L(f)\geq \frac{\text{ord}(\grad
f(\varphi(t))}{\text{ord}(\varphi(t))} \geq\frac{d-w_2}{w_2}.
$$

On the other hand, by the proposition \ref{Hassane}, we obtain that
$$
\rho^{w_2}\parallel \p f\parallel\gtrsim\|\text{grad}_wf(x)\|_w
\gtrsim \rho(x)^d.
$$
Then
$$
\parallel \p f\parallel\gtrsim \rho^{d-w_2}=(\mid x_1\mid^{1/w_1}+\mid
x_2\mid^{1/w_2})^{d-w_2}\gtrsim
\parallel(x_1,x_2)\parallel^{\frac{d-w_2}{w_2}}.
$$
Hence $L(f)\leq \frac{d-w_2}{w_2}$.

This ends the proof of Theorem \ref{main}.

\bigskip
We conclude with several examples.
\begin{example}
Let
$$
f(x,y)=x^3+xy^6 + y^9,
$$
$f$ is  weighted homogenous of type $(9; 3, 1)$ defining an isolated
singularity, since $\{\p_xf=0\}=\{0\}$, then by theorem \ref{main}.
we get $L(f)=6$. Also, for $g(x,y)=x^3-xy^6+ y^9$ can be seen as
weighted homogenous of the same type, but
$\{\p_xg=0\}\nsubseteq\{y=0\}$, hence by theorem \ref{main} we get
$L(g)=8$.

\end{example}
\begin{example}
Let
$$
f(x,y)=y(x^5+xy^{12} + y^{15}),
$$
$f$ is  weighted homogenous of type $(16; 3, 1)$ defining an
isolated singularity, since $\{\p_xf=0\}\subset\{y=0\}$, then by
theorem \ref{main} we get $L(f)=13$. Also, for $g(x,y)=y(x^5-xy^{12}
+ y^{15})$ can be seen as weighted homogenous of the same type, but
$\{\p_xg=0\}\nsubseteq\{y=0\}$, so by theorem \ref{main} we obtain
$L(g)=15$.

\end{example}

\end{document}